%% file: WhatIsSuperrigid.tex
\documentclass[reqno]{conm-p-l}

\pdfoutput=1

\usepackage{relsize}
\usepackage{color}

\input bigset

\usepackage{graphicx}

\definecolor{grey}{rgb}{0.3,0.3,0.3}

\def\glc{\GL(d,\complex)}
\def\glr{\GL(d,\real)}
\renewcommand{\vector}{\overrightarrow}

\def\seg#1#2{\hss$#1$\vrule width 0.4in height 4pt depth-2.5pt $#2$\hss}
\def\sseg#1#2{\hss$#1$\vrule width 0.25in height 4pt depth-2.5pt $#2$\hss}

\renewcommand{\mod}{\mathop{\rm mod}}
\newcommand{\cover}{\widetilde}

\renewcommand{\hat}{\widehat}

\newcommand{\Zar}[1]{\overline{\overline{#1}}}

\DeclareMathOperator{\GL}{GL}
\DeclareMathOperator{\SL}{SL}
\DeclareMathOperator{\Mat}{Mat}

\DeclareMathOperator{\rot}{rot}

\newcommand{\integer}{\mathbb{Z}}
\renewcommand{\natural}{\mathbb{N}}
\newcommand{\rational}{\mathbb{Q}}
\newcommand{\real}{\mathbb{R}}
\newcommand{\complex}{\mathbb{C}}
\newcommand{\F}{\mathbb{F}}

\newcommand{\torus}{\mathbb{T}}
\newcommand{\Euclidean}{\mathbb{E}}

\newcommand{\iso}{\cong}
\newcommand{\compose}{\mathbin{\circ}}
\newcommand{\Id}{\mathord{\mathbb{I}}}

\newcommand{\Rrank}{\mathop{\text{$\real$-rank}}}

\newcommand{\displayline}[1]{\hbox to \textwidth{#1}}

\newcommand{\pref}[1]{{\rm(}\ref{#1}{\rm)}}

\swapnumbers
\numberwithin{equation}{section}

\newtheorem{thm}[equation]{Theorem}
\newtheorem{prop}[equation]{Proposition}

\theoremstyle{definition}
\newtheorem{defn}[equation]{Definition}
\newtheorem{eg}[equation]{Example}
\newtheorem{egs}[equation]{Examples}

\theoremstyle{remark}
\newtheorem{rem}[equation]{Remark}
\newtheorem{notation}[equation]{Notation}
\newtheorem{exer}[equation]{Exercise}
\newtheorem{warn}[equation]{Warning}

\newtheorem*{techrem}{Technical remark}

 \newcounter{step}
 \newenvironment{step}[1][\unskip]{\refstepcounter{step}\em
 \medskip \noindent Step \thestep\ #1.\ }{\unskip\upshape}
 \renewcommand{\thestep}{\arabic{step}}

\begin{document}

\title{What is a superrigid subgroup?}

\author{Dave Witte Morris}
\address{Department of Mathematics and Computer Science, University of Lethbridge,
Lethbridge, Alberta, T1K~3M4, Canada}
\thanks{This article is based on a talk given in various forms at several different universities, and at an MAA Mathfest.
It was written during visits to the University of Chicago and the Tata Institute of Fundamental Research in Mumbai, India; I would like to thank both of these institutions for their generous hospitality.
The writing was partially supported by a research grant from the National Science and
Engineering Research Council of Canada.}

\date{December 14, 2007.} 

\dedicatory{To my friend and mentor Joseph A.~Gallian on his 65th birthday}


\maketitle

\tableofcontents

It is well known that a linear transformation can be defined to have any desired action on a basis.
From this fact, one can show that every group homomorphism
from~$\integer^k$ to~$\real^d$ extends to a homomorphism from~$\real^k$ to~$\real^d$,
and we will see other examples of discrete
subgroups~$H$ of connected groups~$G$, such that the
homomorphisms defined on~$H$ can (``almost") be extended
to homomorphisms defined on all of~$G$.
First, let us see that this is related to a very classical topic in geometry, the study of linkages.

\section{Rigidity of Linkages}

Informally, a \emph{linkage} is an object in $3$-space that is constructed from some finite set of line segments (called ``rods,'' or ``edges'') by attaching endpoints of some of the rods to endpoints of some of the other rods. (That is, a linkage naturally has the structure of a $1$-dimensional simplicial complex.) It is assumed that the rods are rigid (they can neither stretch nor bend), but that the joints that connect the rods are entirely flexible --- they allow the rods to rotate freely, as long as the endpoints remain attached.

\begin{eg}[Hinge]
Construct a linkage with four vertices (or ``joints'') $A,B,C,D$ by putting together two different triangles $ABC$ and $BCD$ with the same base $BC$, as in Figure~\ref{hingefig}(a). 
The angle between the two triangles can be varied continuously, so the object has some flexibility --- it is not \emph{rigid}.
(For example, the hinge can be opened wider, as in Figure~\ref{hingefig}(b).) 
This linkage can reasonably be called a ``hinge\hbox to 0pt{.\hss}''

\begin{figure}[ht]
\hbox to \textwidth{\hfil
\vbox{
\hbox to 2 in{\hfil\includegraphics[scale=0.3]{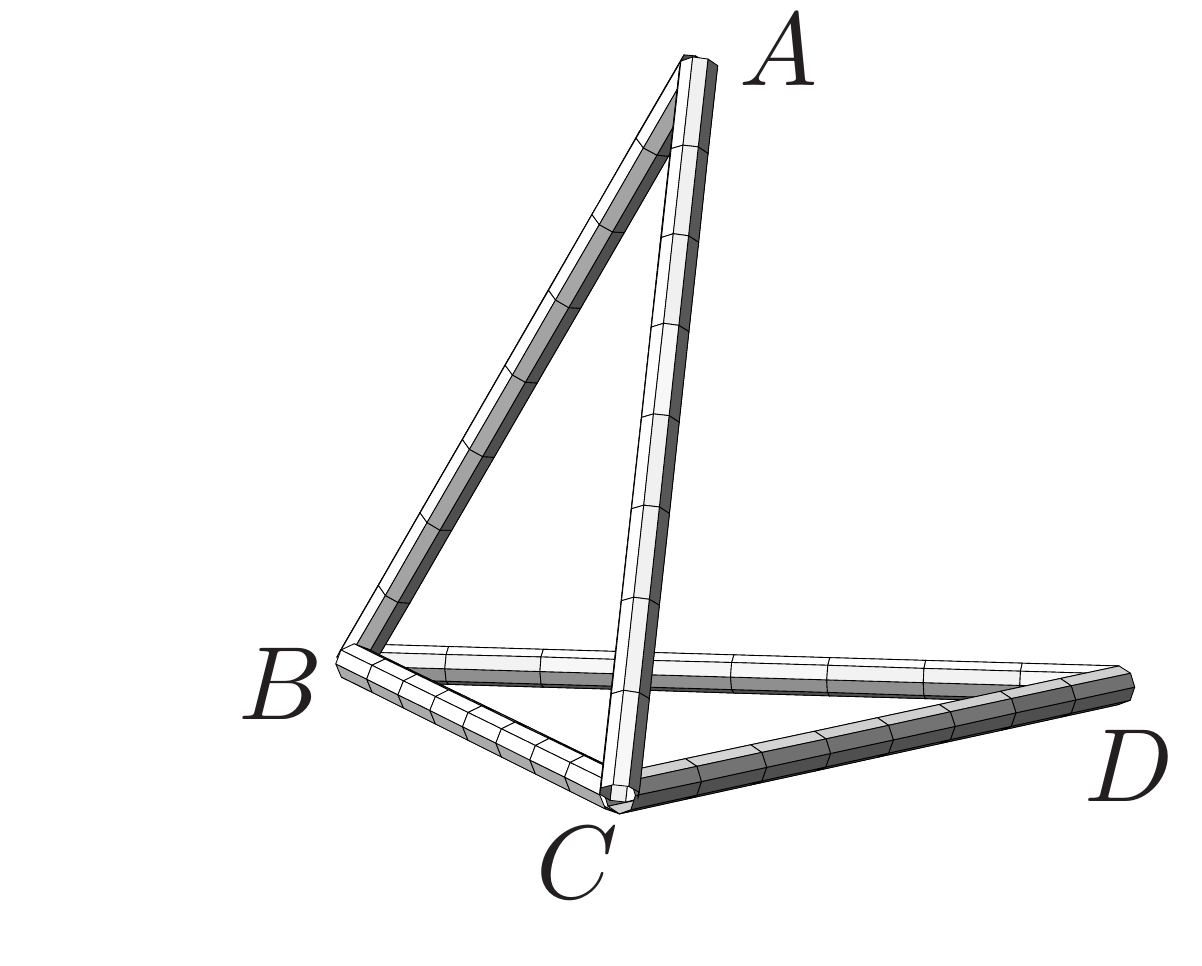}\hfil}
\hbox to 2 in{\hfil(a)\hfil}
}
\vbox{
\hbox to 2 in{\hfil\includegraphics[scale=0.3]{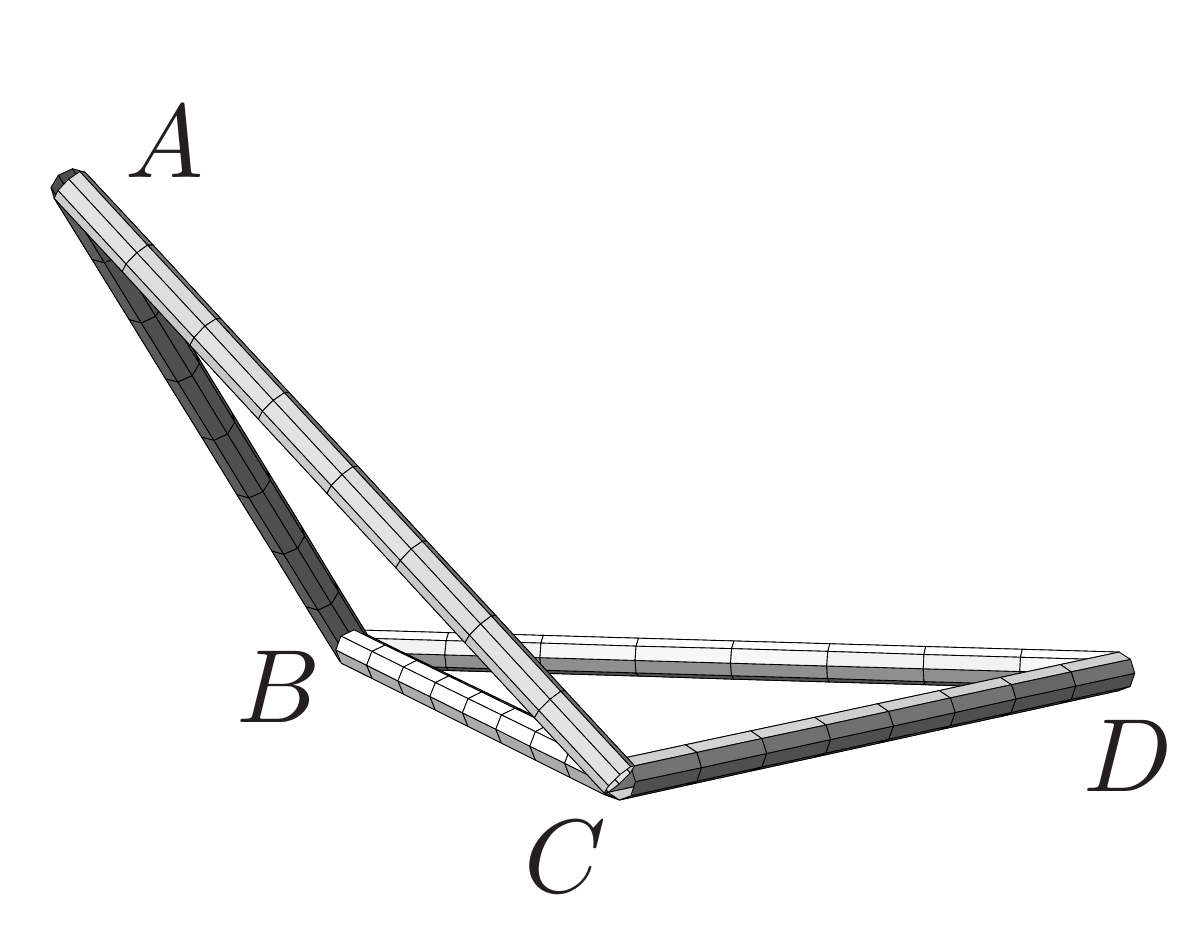}\hfil}
\hbox to 2 in{\hfil(b)\hfil}
}
\hfil}
\caption{The hinge is not rigid, because the angle between the two triangles can be varied continuously, without changing the lengths of the rods in the linkage.}
\label{hingefig}
\end{figure}
\end{eg}

\begin{eg}[Tetrahedron]
Construct a tetrahedron with four vertices $A,B,C,D$ by joining every pair of vertices with an edge, as in Figure~\ref{tetrafig}. 
This object is rigid --- it cannot be deformed.

\begin{figure}[ht]
\includegraphics[scale=0.3]{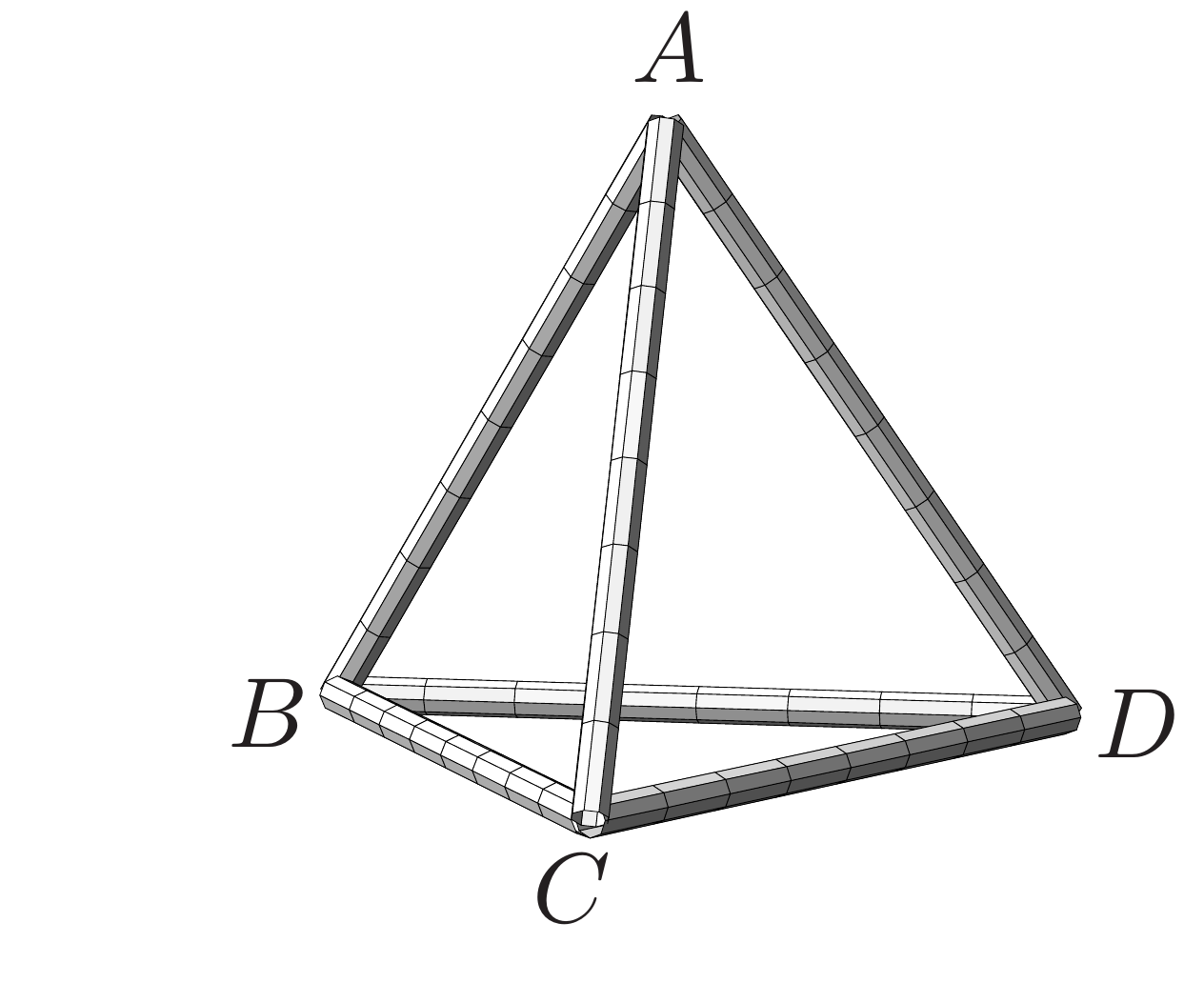}
\caption{A tetrahedron cannot be deformed; it is rigid.}
\label{tetrafig}
\end{figure}
\end{eg}

\begin{eg}[Double tetrahedron]
Add a small tetrahedron $BCDE$ to the bottom of the tetrahedron $ABCD$, as in Figure~\ref{twotetraoutfig}.
The resulting object has no deformations, so it is rigid.
\begin{figure}[ht]
\includegraphics[scale=0.3]{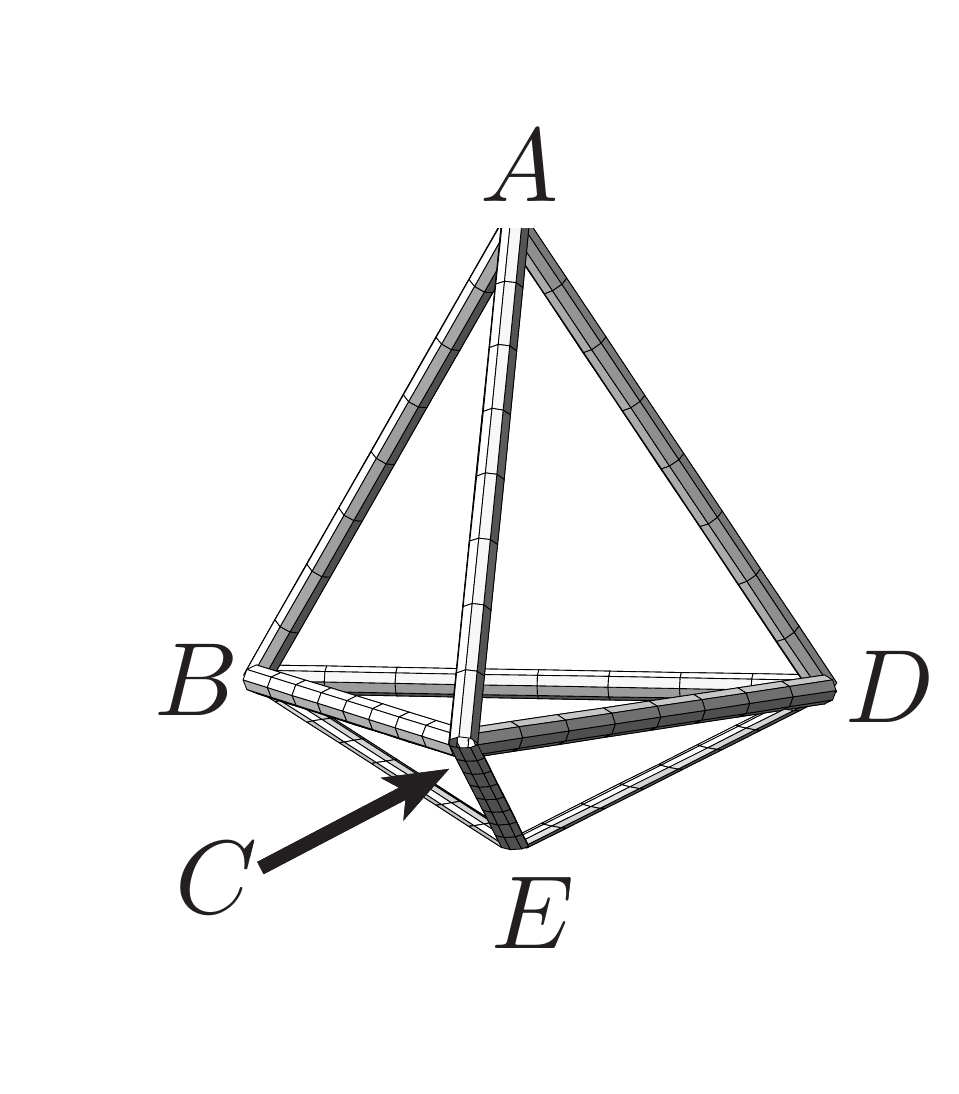}
\caption{Two tetrahedra with a common face form a rigid structure.}
\label{twotetraoutfig}
\end{figure}

However, this double tetrahedron does not have the property that is called \emph{global rigidity}. Namely, suppose:
	\begin{enumerate}
	\item We label each end of each rod with the name of the vertex that joins it to other rods, and then dismantle the linkage. This results in a collection of $9$ rods, which are pictured in Figure~\ref{twotetrarodsfig}.
	\item We then assemble these rods into a linkage, by joining together all vertices that have the same label.
	\end{enumerate}
Unfortunately, the resulting linkage may not be the one we started with; as illustrated in Figure~\ref{twotetrainfig}, the small tetrahedron could be \emph{inside} the larger one, instead of \emph{outside}.

\begin{figure}[ht]
\displayline{
 \seg AB
 \seg AC
 \seg AD
 }
\displayline{
 \seg BC
 \seg BD
 \seg CD
 }
\displayline{
 \sseg BE
 \sseg CE
 \sseg DE
 }
\caption{The double tetrahedron is made up of $9$ rods
(6 long ones and 3 short ones).}
\label{twotetrarodsfig}
\end{figure}  

\begin{figure}[ht]
\includegraphics[scale=0.3]{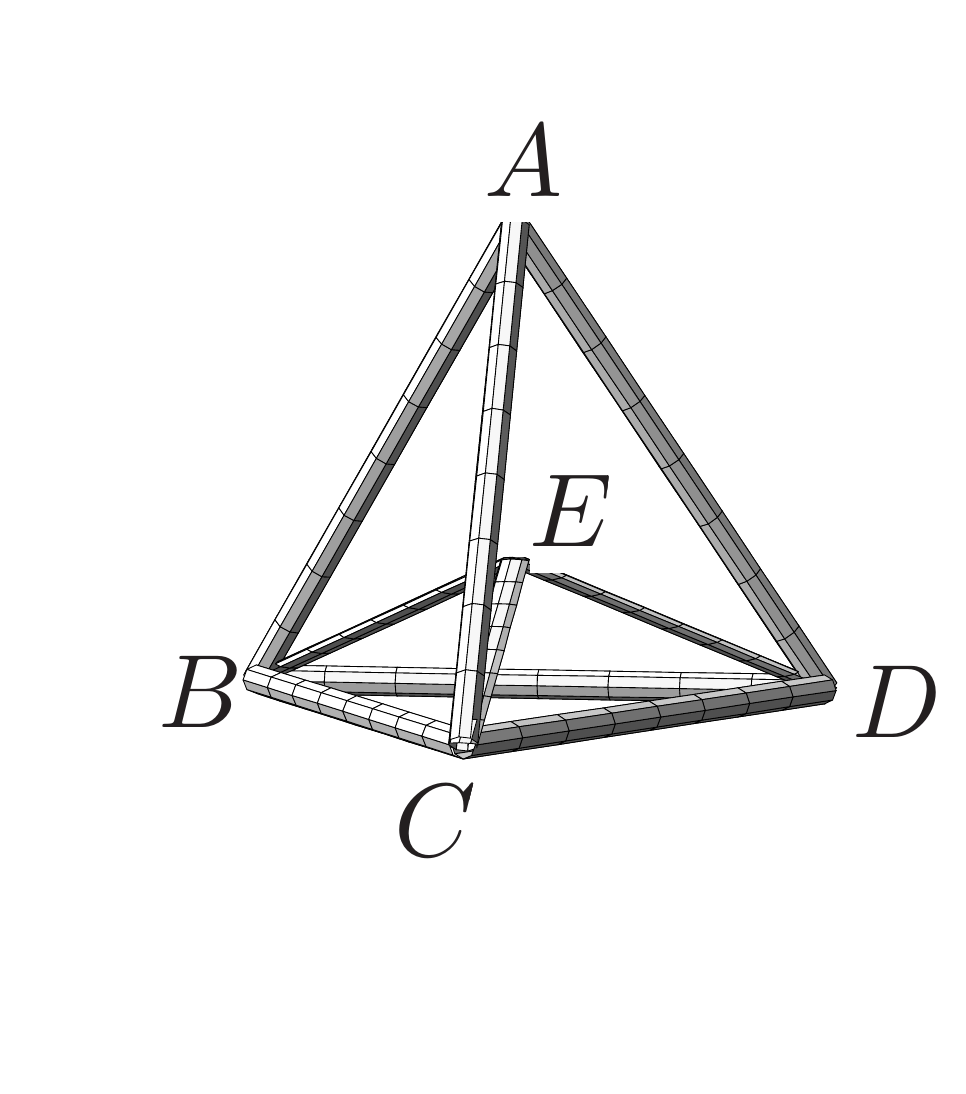}
\caption{The double tetrahedron is not globally rigid: if it has been taken apart, 
   it can be reassembled incorrectly, even if the gluing instructions are followed correctly.}
   \label{twotetrainfig}
   \end{figure}

In summary:
\begin{itemize}
\item The double tetrahedron has no small perturbations. In other words, if it is reassembled, and every rod is close to its correct position, then every rod is in \emph{exactly} the correct position.
So the object is rigid, or, more precisely, ``locally rigid\hbox to 0pt{.\hss}''
\item On the other hand, the double tetrahedron is not rigid in a global sense --- we say that it is not \emph{globally rigid} --- because it can be reassembled incorrectly if we do not assure that the rods are near their correct position.
\end{itemize}
\end{eg}

\begin{eg}
A tetrahedron \emph{is} globally rigid: its geometric structure is completely determined (up to congruence) by the combinatorial data that specify which of the rods are to be joined together.
 \end{eg}

Rigidity and global rigidity are important concepts in geometry, and also in the real world:
	\begin{itemize}
	\item Scaffolds, shelving units, bridges, and many other structures can be viewed as linkages, and they must be designed not to collapse; they must be (locally) rigid.
	\item Furniture and other bulky objects are sometimes shipped in pieces that are to be assembled at the destination, by following instructions of the type ``insert tab~A in slot~B.''
Unless the object is globally rigid, the instructions will be insufficient to guarantee proper assembly.
	\end{itemize}
Thus, it should not be hard to imagine that an analogous notion in other fields would have significant interest to researchers in that area. We will focus on the case of group theory.

\section{The Analogous Notion in Group Theory}

Informally, saying that a linkage~$X$ is globally rigid means that if $Y$ is any linkage that is constructed from rods of the same lengths by using the same combinatorial rules, then $Y$ is congruent to~$X$. Here is a more formal definition:

\begin{defn}
To say that a linkage~$X$ in the Euclidean space~$\Euclidean^3$ is \emph{globally rigid} means that if
	\begin{itemize}
	\item $Y$ is any linkage in~$\Euclidean^3$,
	and
	\item $f \colon X \to Y$ is a combinatorial isomorphism (i.e., $f$ is a bijection that maps each rod in~$X$ isometrically onto a rod in~$Y$),
	\end{itemize}
then $f$ extends to an isometry~$\hat f$ of~$\Euclidean^3$.
\end{defn}

The same idea can easily be adapted to other categories of mathematical objects. For example, replacing $\Euclidean^3$ with a group~$G$, and replacing~$X$ and~$Y$ with subgroups~$H$ and~$K$ of~$G$ yields the following definition, if we realize that an automorphism of~$G$ is the group-theoretic analogue of an isometry of~$\Euclidean^3$.

\begin{eg} \label{ZRigidInREg}
$\real$ is a group (under addition), and $\integer$ is a subgroup. 
If
	\begin{itemize}
	\item $K$ is a subgroup of~$\real$, 
	and 
	\item $\varphi \colon \integer \to K$ is an isomorphism,
	\end{itemize}
then $\varphi$ extends to an automorphism~$\hat\varphi$ of~$\real$.
\end{eg}

\begin{proof}
Let $c = \varphi(1)$ and define 
$\hat\varphi \colon \real \to \real$ by
	$$ \hat\varphi(x) = c x .$$
Then:
	\begin{itemize}
	\item It is obvious that $\hat\varphi$ is a homomorphism. 
	\item Since $\varphi$ is injective, we know 
		$$ c = \varphi(1) \neq \varphi(0) = 0 ,$$
	so $\hat\varphi$ is bijective.
	\item For any $n \in \integer$, we have
		\begin{align*}
		 \hat\varphi(n) 
		&= c n
		&& \text{(definition of $\hat\varphi$)}
		\\&= n \cdot \varphi(1)
		&& \text{(definition of $c$)}
		\\&= \varphi(n)
		&& \text{($\varphi$ is a homomorphism)}
		. \end{align*}
	So $\hat\varphi$ extends~$\varphi$.
	\end{itemize}
 Thus, $\hat\varphi$ is an automorphism of~$\real$ that extends~$\varphi$.
 \end{proof}

In the above example:
	\begin{enumerate} 
	\item The group $\real$ is also a topological space, and the group operations of addition and negation are compatible with the topology (that is, they are continuous); thus, $\real$ is a \emph{topological group}. 
	\item The subgroup $\integer$ is discrete in~$\real$ (i.e., has no accumulation points); so we say that $\integer$ is a \emph{discrete subgroup} of~$\real$.
	\item The homomorphism $\hat\varphi$ is continuous. 
	\end{enumerate}
Thus, $\integer$ is globally rigid in~$\real$, even when we take into account the topological structure of~$\real$:

\begin{defn} \label{GlobalRigDefn}
Let $H$ be a discrete subgroup of a topological group~$G$. Saying $H$ is \emph{globally rigid} in~$G$ means that if
	\begin{itemize}
	\item $K$ is any discrete subgroup of~$G$,
	and
	\item $\varphi \colon H \to K$ is any isomorphism,
	\end{itemize}
then $\varphi$ extends to a continuous automorphism~$\hat \varphi$ of~$G$.
\end{defn}

\section{Definition of Superrigidity}

In the definition of global rigidity \pref{GlobalRigDefn}, the map~$\varphi$ is assumed to be an isomorphism, and its image~$K$ is assumed to be contained in the same group~$G$ that contains~$H$. ``Superrigidity'' is a notion that removes these restrictions. Here is a very elementary example of this. It generalizes Example~\ref{ZRigidInREg}:

\begin{eg} \label{HomoZkToRd}
Suppose $\varphi$ is any group homomorphism from~$\integer^k$ to~$\real^d$.
(That is, $\varphi$ is a function from~$\integer$ to~$\real^d$, and we have $\varphi(m+n) = \varphi (m) + \varphi (n)$.) Then $\varphi$ extends to a continuous homomorphism $\widehat\varphi \colon \real^k \to \real^d$.
\end{eg}

\begin{proof}
Let $e_1,e_2,\ldots,e_k$ be the standard basis of~$\real^k$, so $\{e_1,e_2,\ldots,e_k\}$ is a generating set for the subgroup~$\integer^k$. 
A linear transformation can be defined to have any desired action on a basis, so there is a linear transformation $\hat\varphi \colon \real^k \to \real^d$, such that
	\begin{equation} \label{hatphi(e_i)=phi(e_i)}
	 \text{$\hat\varphi(e_i) = \varphi(e_i)$ \quad for $i = 1,2,\ldots,k$.} 
	 \end{equation}
Then:
	\begin{itemize}
	\item Since $\hat\varphi$ is linear, it is continuous.
	\item Because $\hat\varphi$ is a linear transformation, it respects addition; that is, it is a homomorphism from~$\real^k$ to~$\real^d$. 
	\item From \pref{hatphi(e_i)=phi(e_i)}, we know that $\varphi$ and~$\hat\varphi$ agree on $e_1,e_2,\ldots,e_k$. Thus, since $\{e_1,e_2,\ldots,e_k\}$ generates~$\integer^k$, the two homomorphisms agree on all of~$\integer^k$. In other words, $\hat\varphi$ extends~$\varphi$.
	\end{itemize}
So $\hat\varphi$ is a continuous automorphism that extends~$\varphi$.
\end{proof}

In short:
\begin{equation} \label{ZkToRdExtends}
\begin{matrix}
\text{\it Every group homomorphism from~$\integer^k$ to~$\real^d$} \hfill
\\ \text{\it extends to a continuous homomorphism from~$\real^k$ to~$\real^d$.} \hfill
\end{matrix}
\end{equation}
However, because this observation deals only with abelian groups, it is rather trivial.
A superrigidity theorem is a result of similar flavor that deals with more interesting groups.
Namely, instead of only homomorphisms into the abelian group~$\real^d$, it is much more interesting to look at homomorphisms into matrix groups. (Any such homomorphism is called a \emph{group representation}, and the study of these representations is a major part of group theory.) 

Let us be more precise:

\begin{notation}
$\glr =\{\, \text{$d  \times d$ invertible matrices with real entries} \,\}$.
\end{notation}

It is important to note that $\glr$ is a group under multiplication. Furthermore, $\real^k$ is a subgroup of $\glr$ (if $d > k$). For example,
	$$\real ^3 \iso \begin{bmatrix}
 1&0&0&*\\
 0&1&0&*\\
 0&0&1&*\\
 0&0&0&1\\
 \end{bmatrix}
 \subset \GL(4,\real) .$$ 
So any homomorphism into $\real^d$ 
   can be thought of as a homomorphism into $\GL(d+1,\real)$.

Unfortunately, \pref{ZkToRdExtends} does not remain valid if we replace $\real^d$ with $\glr$:

\begin{eg}
Suppose $\varphi$ is a group homomorphism from~$\integer$ to $\glr$. That is, $\varphi$ is a function from~$\integer$ to $\glr$, and we have
	$$ \varphi (m+n) = \varphi (m) \cdot \varphi (n) .$$
It \emph{need not} be the case that $\varphi$ extends to a continuous homomorphism from~$\real$ to $\glr$.
\end{eg}

\begin{proof}[Proof by contradiction]
 Suppose there is a continuous homomorphism $\widehat\varphi  \colon \real \to \glr$, such that $\widehat\varphi(n) = \varphi(n)$, for all $n \in \integer$.

Consider the composition $\det \compose \widehat\varphi$. Note that:
	\begin{itemize}
	\item Since the determinant of any invertible matrix is nonzero, we see that 
$\det \compose \widehat\varphi$ is a function from~$\real$ to~$\real^\times$ (where $\real^\times$ is the set of nonzero real numbers).
	\item Since homomorphisms map the identity element of the domain group to the identity element of the image, we have $\varphi(0) = \Id$ (the identity matrix). Hence, 
		$$\det \bigl( \varphi(0) \bigr) = \det(\Id) = 1 > 0 .$$
	\item Since the composition of continuous functions is continuous, and the continuous image of a connected set is continuous, we know $\det \bigl( \hat\varphi (\real) \bigr)$ is connected.
	\end{itemize}
Therefore, $\det \bigl( \hat\varphi (\real) \bigr)$ is a connected subset of $\real^\times$ that contains the number~$1$. So $\det \bigl( \hat\varphi (\real) \bigr) \subset \real^+$. In particular,
	$ \det \bigl( \hat\varphi (1) \bigr) > 0$. Therefore
		$$ \det \bigl(\varphi (1) \bigr) = \det \bigl( \hat\varphi (1) \bigr) > 0 . $$
But $\varphi$ is an arbitrary homomorphism from~$\integer$ to $\glr$, and it need not be the case that $ \det \bigl(\varphi (1) \bigr)  > 0$. (Namely, for any $A \in \glr$, we may let $\varphi(n) = A^n$. If $\det A < 0$, then $\det \bigl(\varphi (1) \bigr) = \det A < 0$.) This is a contradiction.
\end{proof}

The above counterexample is based on the possibility that $\det\bigl(\varphi (1)\bigr) < 0$. However, for any~$n$, we have
	$$\det\bigl(\varphi (2n)\bigr) 
	  = \det\bigl(\varphi (n + n)\bigr) 
	  = \det\bigl(\varphi (n) \cdot \varphi(n) \bigr)
	  = \Bigl(\det\bigl(\varphi (n)\bigr) \Bigr)^2
	  > 0 .$$
Thus, this possibility does not arise if we restrict our attention to even numbers.
That is, in defining the extension~$\widehat\varphi$, which interpolates a nice curve through the given values at points of~$\integer$, we may have to ignore the values at odd numbers, and only match the values of~$\varphi$ at even numbers. An illustration of this is in Figure~\ref{ignoreoddsfig}.

\begin{figure}
\includegraphics[scale=0.3]{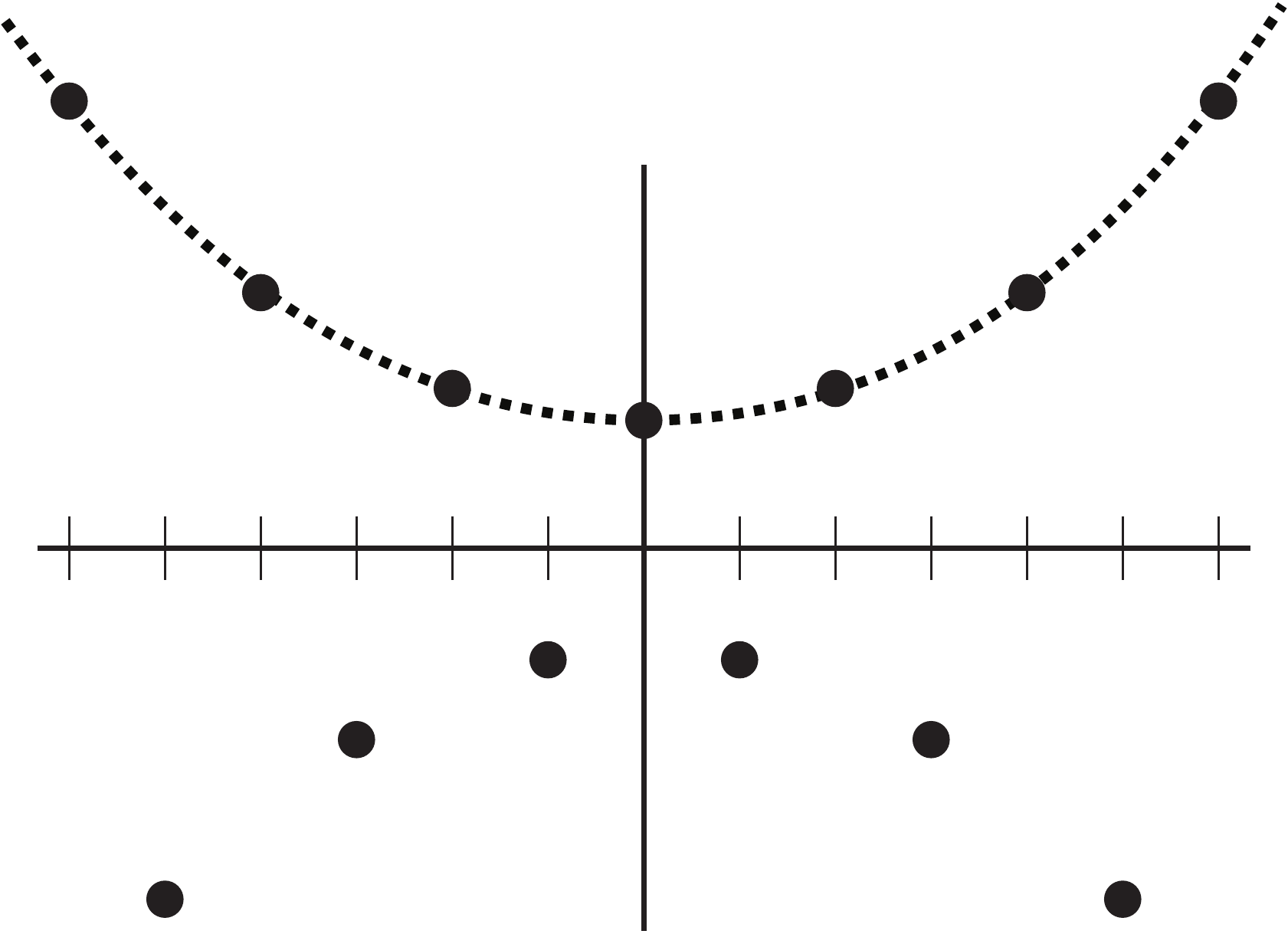}
\caption{It may be necessary to ignore the values at odd numbers when interpolating.}
\label{ignoreoddsfig}
\end{figure}

One can imagine that, analogously, there might be situations where it is necessary to restrict attention, not to multiples of~$2$, but to multiples of some other integer~$N$. A group theorist may observe that
	$$\text{$\{\, \hbox{multiples of~$N$} \,\}$ is a subgroup of~$\integer$ that has finite index.} $$
Thus, in group-theoretic terms, the upshot of the preceding discussion is that we may need to
restrict our attention to a finite-index subgroup. 

The need to pass to a finite-index subgroup happens so often in the theory of infinite groups that there is a name for it: a property holds \emph{virtually} if it becomes true when our attention is restricted to a finite-index subgroup.

\begin{eg} \ 
\begin{enumerate}
\item To say that $G$ is \emph{virtually abelian} means that some finite-index subgroup of~$G$ is abelian.
\item If $G$ is a topological group, then, to say that $G$ is \emph{virtually connected} means that some finite-index subgroup of~$G$ is connected.
\end{enumerate}
\end{eg}

\begin{exer}
What does it mean to say that $G$ is \emph{virtually finite}?
\end{exer}

In this vein, we make the following definition:

\begin{defn}
Suppose
	\begin{itemize}
	\item $H$ is a discrete subgroup of a topological group~$G$,
	\item $\varphi \colon H \to \glr$ is a homomorphism,
	and
	\item $\hat\varphi \colon G \to \glr$ is a continuous homomorphism.
	\end{itemize}
We say $\hat\varphi$ \emph{virtually extends}~$\varphi$ if there is a finite-index subgroup~$H'$ of~$H$, such that $\hat\varphi(h) = \varphi(h)$, for all $h \in H'$.
\end{defn}

Although the proof is not obvious, it turns out that homomorphisms defined on~$\integer^k$ do virtually extend to be defined on all of~$\real^k$:

\begin{prop} \label{ZkAlmostSuperrig}
Suppose $\varphi$ is a group homomorphism from~$\integer^k$ to $\glr$. 
Then $\varphi$ virtually extends to a continuous homomorphism $\hat \varphi \colon G \to \glr$.
 \end{prop}
 
 Unfortunately, this result is usually not useful, because it does not tell us anything about the image of $\hat\varphi$ (other than that it is contained in $\glr$). In practice, if all of the matrices in $\varphi(\integer^k)$ have some nice property, then it is important to know that the matrices in $\hat\varphi(\real^k)$ also have this property. That is, if we have control on the image of~$\varphi$, then we would like to have control on the image of~$\hat\varphi$.
 
\begin{eg} \ 
	\begin{enumerate}
	\item If all of the matrices in $\varphi(\integer^k)$ have determinant~$1$, then all of the matrices
in $\widehat\varphi(\real^k)$ should have determinant~$1$.
	\item If all of the matrices in $\varphi(\integer^k)$ commute with some particular matrix~$A$, then all of the matrices in $\widehat\varphi(\real)$ should commute with~$A$.
	\item If all of the matrices in $\varphi(\integer^k)$ fix a particular vector~$v$, then all of the matrices in $\widehat\varphi(\real^k)$ should fix~$v$.
	\item Let $R = \begin{bmatrix}
 1&0&0&*\\
 0&1&0&*\\
 0&0&1&*\\
 0&0&0&1\\
 \end{bmatrix} \iso \real^3$. If $\varphi(\integer^k) \subset  R$, then it should be the case that $\hat\varphi(\real^k) \subset R$. One needs to know this in order to derive Example~\ref{HomoZkToRd} as a corollary of a result like Proposition~\ref{ZkAlmostSuperrig}.
	\end{enumerate}
\end{eg}

\begin{rem}
The problem that arises here is illustrated by the classical theory of Lagrange interpolation. This theorem states that if $(x_0,y_0), (x_1,y_1), \ldots, (x_n,y_n)$ are any $n+1$ points in the plane (with $x_i \neq x_j$ whenever $i \neq j$), then there is a polynomial curve 
  	$$y = f(x) = a_n x^n + a_{n-1}x^{n-1} + \cdots + a_0$$
of degree~$n$ that passes through all of these points. (It is easy to prove.) Unfortunately, however, even if the specified values $y_0,y_1,\ldots,y_n$ of $f(x)$ at the points $x_0,x_1,\ldots,x_n$ are well controlled (say, all are less than~$1$ in absolute value), it may be the case that $f(x)$ takes extremely large values  at other values of~$x$ that are between $x_0$ and~$x_n$, as illustrated in Figure~\ref{LagrangeInterpFig}.

\begin{figure}[ht]
\includegraphics[scale=0.3]{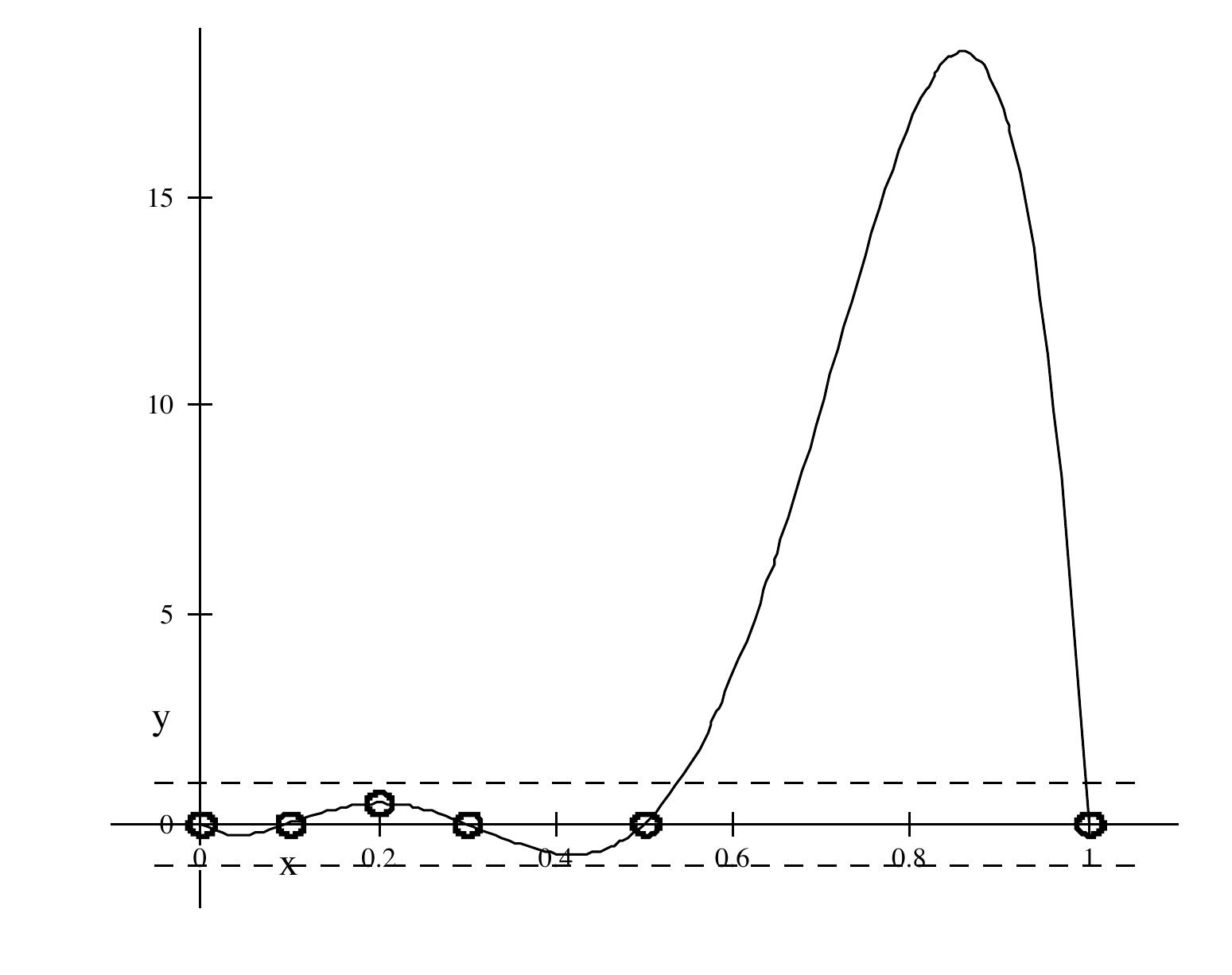}
\caption{The 6 given points all lie in a small band around the $x$-axis, but the quintic curve that interpolates between them travels far from the $x$-axis.}
 \label{LagrangeInterpFig}
\end{figure}

Linear interpolation does not suffer from this defect; all of the points of the interpolating curve will lie in the convex hull of the given points.
 \end{rem}

In order to guarantee that having control on the values of~$\varphi$ will guarantee that we have control on the values of~$\hat\varphi$, we will require $\hat\varphi(H)$ to be contained in a certain subgroup $\Zar{H}$ of $\glr$ that is closely related to~$\varphi(H)$. (This subgroup is called the ``Zariski closure'' of $\varphi(H)$.) 

The formal definition of the \emph{Zariski closure}~$\Zar H$ of a subgroup~$H$ of $\glr$ is not important for our purposes, if one simply accepts that it is, in a certain precise sense, the smallest natural, closed, virtually connected subgroup of $\glr$ that contains~$H$. 
It can be thought of as the group-theoretic analogue of a convex hull.

\begin{defn} \label{StrictSuperrigDefn}
Let $H$ be a discrete subgroup of a topological group~$G$. Saying $H$ is \emph{{\rm(}strictly{\rm)} superrigid} in~$G$ means, for all~$d$, that 
if $\varphi \colon H \to \glr$ is any homomorphism,
then $\varphi$ virtually extends to a continuous automorphism $\hat \varphi \colon G \to \glr$, such that $\hat\varphi(G) \subset \Zar{\varphi(H)}$.
\end{defn}

We have the following example:

\begin{prop} \label{ZkSuperrigProp}
$\integer^k$ is strictly superrigid in~$\real^k$. 
\end{prop}

\begin{proof}[Proof {\rm(optional)}]
For simplicity, let us assume $k = 1$; thus, we wish to show $\integer$ is strictly superrigid in~$\real$. 

Given a homomorphism $\varphi \colon \integer \to \glr$, let $Y = \Zar{\varphi(\integer)}$, and let $Y^\circ$ be the connected component of~$Y$ that contains~$e$ (so $Y^\circ$ is a closed subgroup of~$Y$). Since the Zariski closure~$Y$ has only finitely many connected components, there is some nonzero $m \in \integer$, such that $\varphi(m) \in Y^\circ$.

Since $\varphi(\integer)$ is abelian, it is not difficult to see that its Zariski closure is also abelian. So $Y^\circ$ is an abelian Lie group; therefore, the universal cover~$\cover{Y^\circ}$ of~$Y^\circ$ is a simply connected, abelian Lie group. One can show that this implies $\cover{Y^\circ}$ is isomorphic to~$\real^n$, for some~$n$. So there is no harm in assuming that $\cover{Y^\circ}$ is actually equal to~$\real^n$.
	\begin{itemize}
	\item Let $\pi \colon \real^n \to Y^\circ$ be the covering map with $\pi(0) = e$, so $\pi$ is a continuous homomorphism.
	\item Choose some $\vector y \in \real^n$, such that $\pi(\vector y) = \varphi(m)$.
	\item Define $\cover\varphi \colon \real \to \real^n$ by 
		$$\cover\varphi(x) = \frac{x}{m} \vector y .$$
	\item Let $\hat\varphi \colon \real \to Y^\circ$ be the composition $\pi \circ \cover\varphi$.
	\end{itemize}
Then:
	\begin{itemize}
	\item $\hat\varphi$ is a composition of continuous homomorphisms, so it is a continuous homomorphism.
	\item We have
		$$\hat\varphi(m) 
		= \pi \bigl( \cover\varphi(m) \bigr)
		= \pi \left(  \frac{m}{m}\vector y \right)
		= \pi \bigl(  \vector y \bigr)
		= \varphi(m) ,$$
	so $\hat\varphi$ is equal to $\varphi$ on the entire cyclic subgroup generated by~$m$. Since $m \neq 0$, this is a finite-index subgroup of~$\integer$.
	\item We have 
		$$\hat\varphi(\real) 
		= \pi \bigl( \cover\varphi(\real) \bigr) 
		\subset \pi(\real^n) 
		= Y^\circ
		\subset Y 
		= \Zar{\varphi(\integer)}
		.$$
	\end{itemize}
Thus, $\hat\varphi$ is a continuous homomorphism that virtually extends~$\varphi$, such that $\hat\varphi(\real) \subset \Zar{\varphi(\integer)}$.
\end{proof}

\begin{rem} \label{NotStrictRem}
The term ``strictly'' is used in Definition~\ref{GlobalRigDefn} to indicate that $\hat\varphi$ is required to be exactly equal to~$\varphi$ (on a finite-index subgroup). If we drop this modifier, then it means that we do not require exact equality; instead, we allow an error that is uniformly bounded (on a finite-index subgroup of~$H$). That is, we require $\hat\varphi(h) = \varphi(h) \pmod{K}$, where $K$ is some compact group.
\end{rem}

\subsection{Definition of the Zariski closure}

The concept of Zariski closure is taken from algebraic geometry. In that field, one works only with polynomials (and rational functions), not with more general continuous functions, and the notion of Zariski closure is a reflection of this.
For the reader who wants details, we provide the full definition; others are welcome to skip ahead to the following section. 

\begin{rem}
In linear algebra, one works only with linear functions, and the definition of \emph{linear span} is a reflection of this:
	\begin{itemize}
	\item A subset~$V$ of $\real^d$ is a \emph{linear subspace} if it is the set of solutions of a collection of linear equations; more precisely, this means there are linear functionals $\lambda_i \colon \real^d \to \real$, such that 
	$$ v \in V
	\quad\iff\quad
	\text{$\lambda_i(v) = 0$, for all~$i$} .$$
	\item The \emph{linear span} $\langle S \rangle$ of a subset~$S$ of~$\real^d$ is the unique smallest linear subspace of~$\real^d$ that contains~$S$.
	\end{itemize}
The Zariski closure is perfectly analogous, replacing ``linear functional on~$\real^d$'' with ``polynomial function on $\glr$.''
\end{rem}

\begin{defn} \ 
	\begin{itemize}
	\item The collection $\Mat_{d \times d}(\real)$ of all $d \times d$ matrices can naturally be identified with $\real^{d^2}$. A function $P \colon \Mat_{d \times d}(\real) \to \real$ is said to be a \emph{polynomial} if becomes a polynomial (in $d^2$ variables) on~$\real^{d^2}$ after making this identification.
	\item The group $\glr$ can be embedded in the group $\SL(d+1,\real)$ of $(d+1) \times (d+1)$ matrices of determinant~$1$, via the map
		$$ \rho(A) = 
		\begin{bmatrix} 
		& & &0 \\ 
		&\hbox{\Huge $A$}& &\vdots \\
		&& & 0 \\
		0 & \cdots & 0 & 1/\det A 
		\end{bmatrix} 
		.$$
	A function $f \colon \glr \to \real$ is said to be a \emph{polynomial} if there exists a polynomial $P \colon \Mat_{(d+1) \times (d+1)}(\real) \to \real$, such that
		$$ f(g) = P \bigl( \rho(g) \bigr) ,$$
	for all $g \in \glr$.
	\item A subset $V$ of $\glr$ is \emph{Zariski closed} if it is the set of solutions of a collection of polynomial equations; more precisely, this means there are polynomial functions $f_i \colon \glr \to \real$, such that 
	$$ v \in V
	\quad\iff\quad
	\text{$f_i(v) = 0$, for all~$i$} .$$
	\item The \emph{Zariski closure} $\Zar{V}$ of a subset~$V$ of~$\glr$ is the unique smallest Zariski closed subset of $\glr$ that contains~$V$.
	\end{itemize}
\end{defn}

\begin{rem} \ 
	\begin{enumerate}
	\item If $V$ is a subgroup of $\glr$, then $\Zar V$ is also a subgroup of $\glr$.
	\item $\Zar V$ is a closed subset of $\glr$.
	\item $\Zar V$ has only finitely many connected components.
	\end{enumerate}
The first two of these observations are not difficult to prove. The third is rather difficult, but it is a generalization of the obvious fact that a univariate polynomial $f(x)$ can have only finitely many zeroes.
\end{rem}

\section{Examples of Superrigid Subgroups}

Proposition~\ref{ZkSuperrigProp} tells us that $\integer^k$ is strictly superrigid in~$\real^k$, and we will now see other examples of superrigid subgroups. 

Let us first specify the type of group~$G$ that will be considered:

\begin{defn}
We say $G$ is a \emph{Lie group} if it is a closed, connected subgroup of $\glr$, for some~$d$.
\end{defn}

\begin{eg}
$\real^d$ is (isomorphic to) a Lie group.
\end{eg}

\begin{warn}
Other authors have a less restrictive definition of ``Lie group,'' but this will suffice for our purposes.
\end{warn}

Now we wish to describe the subgroups~$H$ of~$G$ that complete the analogy
	$$ \text{$\integer^k$ is to~$\real^k$ \quad as \quad $H$ is to~$G$}. $$
Here are the basic properties of~$\integer^k$:
	\begin{enumerate}
	 \item $\integer ^k$ is a discrete subgroup of~$\real^k$.
	 \item The quotient space $\real^k / \integer^k$ is compact.
	 (Indeed, $\real^k/\integer^k$ is the $k$-torus $\torus^k$, which is well known to be compact.)
	\end{enumerate}
The second of these properties can be restated as the assertion that there is a compact subset of $\real^k$ that contains a representative of every coset of~$\integer^k$. 
Thus, $\integer^k$ is a (cocompact) lattice, in the following sense:

\begin{defn}
Suppose $H$ is a discrete subgroup of a Lie group~$G$. We say $H$ is a (cocompact) \emph{lattice} in~$G$ if  there is a compact subset of~$G$ that contains a representative of every coset of~$H$. 
\end{defn}

\begin{rem}
Cocompact lattices suffice for most of our purposes, but we will sometimes allow~$H$ to satisfy the condition that some set of coset representatives has finite measure.
Since every compact set has finite measure, but not every set of finite measure is compact, this is a more general condition.
\end{rem} 

For the moment, let us assume that $G$ is solvable:

\begin{defn} 
Let $G$ be a connected subgroup of $\glc$.
We say $G$ is \emph{solvable} if and only if it is upper triangular
	$$G \subset \begin{bmatrix}
 \complex^{\times}&\complex&\complex\\
 0&\complex^{\times}&\complex\\
 0&0&\complex^{\times}
\end{bmatrix},$$
or can be made so by a change of basis.
\end{defn}

\begin{rem}
The following example is the base case of an inductive proof that if we restrict our attention only to connected groups, then the above definition agrees with the usual definition of solvable groups in terms of chains of normal subgroups with abelian quotient groups.
\end{rem}

\begin{eg}
All abelian groups are solvable.
\end{eg}

\begin{proof}
It is well known that every matrix can be triangularized over any algebraically closed field, such as~$\complex$.
(That is, there is a change of basis that makes the matrix upper triangular.)
This implies that every cyclic group is solvable.

More generally, it is not difficult to show that any collection of pairwise commuting matrices can be simultaneously triangularized. (That is, there is a single change of basis that makes all of the matrices upper triangular.) This implies that every abelian group is solvable.
 \end{proof}

\begin{egs} \ 
\begin{enumerate}
\item Let 
$$ \text{$G_1 = \begin{bmatrix}
 1&0&0&\real\\
 0&1&0&\real\\
 0&0&1&\real\\
 0&0&0&1
 \end{bmatrix}
 \iso \real^3$
 and
$H_1  = \begin{bmatrix}
 1&0&0&\integer\\
 0&1&0&\integer\\
 0&0&1&\integer\\
 0&0&0&1
 \end{bmatrix}
 \iso \integer^3$.} $$
 Then:
	 \begin{itemize}
	 \item $H_1 \iso \integer^3$ and $G_1 \iso \real^3$, so it is clear that $H_1$ is a lattice in~$G_1$.
 	\item We have already seen that $H_1$ is strictly superrigid in~$G_1$.
	\item $G_1$ is the obvious connected group containing~$H_1$, so $\Zar{H_1} = G_1$. 
	\end{itemize}

\item Let 
$$ \text{$G_2 = \begin{bmatrix}
 1&\real&\real&\real\\
 0&1&\real&\real\\
 0&0&1&\real\\
 0&0&0&1
 \end{bmatrix}$
and
$H_2  = \begin{bmatrix}
 1&\integer&\integer&\integer\\
 0&1&\integer&\integer\\
 0&0&1&\integer\\
 0&0&0&1
 \end{bmatrix}$.} $$
 Then:
	 \begin{itemize}
	 \item It is not difficult to see that $H_2$ is a lattice in~$G_2$. Namely, if we let $I = [0,1]$ be the unit interval then
	 $$ \begin{bmatrix}
 1&I&I&I\\
 0&1&I&I\\
 0&0&1&I\\
 0&0&0&1
 \end{bmatrix}$$
 is a compact set that contains a representative of every coset.

\item Our main result, to be stated below, will show that $H_2$ is strictly superrigid in~$G_2$.
\item $G_2$ is the obvious connected group containing~$H_2$, so $\Zar{H_2} = G_2$. 
\end{itemize}



\item Let
$$ \text{$G_3 = \begin{bmatrix}
 1&\real&\complex\\
 0&1&0\\
 0&0&1
 \end{bmatrix}$ 
 and
 $H_3 = \begin{bmatrix}
 1&\integer&\integer +\integer i\\
 0&1&0 \\
 0&0&1
 \end{bmatrix}$.} $$
Then:
	\begin{itemize}
	\item It is not difficult to see that $H_3$ is a lattice in~$G_3$. Indeed, $H_3 \iso \integer^3 \subset \real^3 \iso G_3$.
	\item We know that $H_3$ is strictly superrigid in~$G_3$.
 	\item $G_3$ is the obvious connected group containing~$H_3$, so  $\Zar{H_3} = G_3$. 
	\end{itemize}

\end{enumerate}
\end{egs}

\begin{eg}
Let
$$ \text{$G' = \bigset{\begin{bmatrix}
 1&t&\complex\\
 0&1&0\\
 0&0&e^{2\pi it}
 \end{bmatrix}}
 {t \in \real}$
 and
 $H' = \begin{bmatrix}
 1&\integer&\integer +\integer i\\
 0&1&0 \\
 0&0&1
 \end{bmatrix}$.} $$
 Unlike our previous examples, the matrix entries of elements of~$G'$ cannot be chosen independently of each other: the $(1,2)$-entry of any element of~$G'$ uniquely determines its $(3,3)$-entry. 
  However, the relation between these entries is defined by a transcendental function, not a polynomial, so, as far as an algebraic geometer is concerned, these entries have no correlation at all. This means that in the Zariski closure of~$G'$, these entries become decoupled and can be chosen independently. Thus,
  	$$ \Zar{G'} = \begin{bmatrix}
 1&\real&\complex\\
 0&1&0\\
 0&0&\torus
 \end{bmatrix} .$$

 When the $(1,2)$-entry~$t$ of an element of~$G'$ is an integer, the $(3,3)$-entry $e^{2\pi i t}$ is~$1$, so we see that $H' \subset G'$. In fact, it is not difficult to see that $H'$ is a lattice in~$G'$. 

On the other hand, we have $H' = H_3 \subset G_3$, so 
$H'$ is also a lattice~$G_3$. Furthermore, we have
 	$$ \Zar{H'} = \Zar{H_3} = G_3 \neq \Zar{G'} .$$
These observations can be used to show that $H'$ is \emph{not} strictly superrigid in~$G'$.
\end{eg}

\begin{prop} \label{NotSuperEgProp}
$H'$ is not strictly superrigid in~$G'$.

In particular,  the inclusion map $\varphi  \colon H' \hookrightarrow \GL(3,\complex)$
does not extend to a continuous homomorphism $\widehat\varphi  \colon G' \to \Zar{H'}$.
\end{prop}

\begin{proof}
Note that
	$$\Zar{H'} = \Zar{H_3} = G_3 $$
is abelian. Therefore, $\hat\varphi$ must be trivial on the entire commutator subgroup $[G',G']$ of~$G'$. 
We have
	$$ [G',G'] 
	= \begin{bmatrix}
 1&0&\complex\\
 0&1&0\\
 0&0&1
 \end{bmatrix} 
 \supset \begin{bmatrix}
 1&0&\integer + \integer i\\
 0&1&0\\
 0&0&1
 \end{bmatrix} 
 ,$$
 so the supposed extension $\hat\varphi$ is trivial on some nontrivial elements of~$H'$. This contradicts the fact that $\varphi$, being an inclusion, has trivial kernel.
  \end{proof}

\begin{rem} \ 
\begin{enumerate}
\item $H_3$ is a strictly superrigid lattice in~$G_3$, but we constructed $G'$ by adding some rotations to~$G_3$ that $H_3$ knows nothing about. A homomorphism defined on~$H_3$ will extend to~$G_3$, but it need not be compatible with the additional rotations that appear in~$G'$.
\item One can show that the above example is typical: it is always the case that if $\Zar{H} \not= \Zar{G}$, then some of the rotations associated to elements of~$G$ do not come from rotations associated to~$H$. Roughly speaking, the concept of ``associated rotation'' can be defined by
	$$\rot\begin{bmatrix}
 \alpha &*\\
 0&\beta
 \end{bmatrix}
  = \begin{bmatrix}
{\alpha }/{ |\alpha |} &0\\
 0&{\beta }/{ |\beta |} 
 \end{bmatrix}.$$
\end{enumerate}
\end{rem}

In general, if $\Zar{H} \neq \Zar{G}$, then the natural connected subgroup containing~$H$ is not~$G$, but some other group; there are parts of~$G$ that have nothing to do with~$H$. A homomorphism defined on~$H$ cannot be expected to know about the structure in this part of~$G$, so there is no reason to expect the homomorphism to be compatible with this additional structure.

The above considerations might lead one to believe that if $\Zar{H} \neq \Zar{G}$, then $H$ is \emph{not} strictly superrigid in~$G$. This conclusion is correct in spirit, but there is a technical complication\footnote
{A given group~$G$ can usually be embedded into $\GL(d,\real)$ in many different ways, and $\Zar{H}$ may be equal to $\Zar{G}$ for some of these embeddings, but not others. The canonical matrix representation that can be used is the so-called ``adjoint representation,'' which is not an embedding: its kernel is the center $Z(G)$, and the Zariski closure should be calculated modulo this kernel.}
 that leads to the fine print in the statement of the following result.
 The reader is invited to simply ignore this fine print.

\begin{prop}
If $H$ is strictly superrigid in~$G$, then
 $\Zar{H} = \Zar{G}$
 \quad \hbox{\color{grey}\smaller[2]\rm({$\mod{\Zar{Z(G)}}$}).} 
\end{prop}

By passing to the universal cover, let us assume that $G$ is simply connected.
Then the converse of the above proposition is true for solvable groups:

\begin{thm} \label{SuperSolvThm}
A lattice $H$ in a simply connected, solvable Lie group~$G$ is strictly superrigid
 if and only if $\Zar H =\Zar G$
  \quad \hbox{\color{grey}\smaller[2]\rm({$\mod{\Zar{Z(G)}}$}).} 
 \end{thm}

This theorem provides a complete characterization of the strictly superrigid lattices in the solvable case.

\subsection{Brief discussion of groups that are not solvable}
 
An extensive structure theory has been developed for Lie groups. Among other things, it is known that these groups can be classified into three basic types:
\begin{itemize}
 \item solvable (e.g., $\real^k$),
 \item semisimple (e.g., $\SL(k,\real)$),
 or
 \item a combination of the above (e.g., $G = \real^k \times \SL(k,\real)$).
 \end{itemize}

In the preceding section, we constructed lattices in solvable groups by taking the integer points in~$G$. For example, $\integer^k$ is a lattice in~$\real^k$. (One might note that, in the case of $H_3$, we used Gaussian integers, not only the ordinary integers.) It turns out that the same construction can be applied to many groups that are not solvable. For example, $\SL(k,\integer)$ is a lattice in $\SL(k,\real)$.

It is known that if $G$ is a combination of a solvable group and a semisimple group, then, roughly speaking, any lattice in~$G$ also has a decomposition into a solvable part and a semisimple part. For example, $ \integer^k \times \SL(k,\integer)$ is a lattice in $\real^k \times \SL(k,\real)$.

The following theorem shows that deciding whether or not~$H$ is superrigid reduces to the same question about its semisimple part:

\begin{thm} \label{SuperComboThm}
A lattice $H$ in a simply connected Lie group~$G$ is superrigid if and only if
 \begin{itemize}
 \item the semisimple part of~$H$ is superrigid,
 and
  \item $\Zar{H} =\Zar{G}$
 \qquad
{\color{grey}\smaller[2]\rm($\mod \Zar{Z(G)}
  \cdot K$, where $K$ is a compact, normal subgroup of~$\Zar{G}$).}
  \end{itemize}
   \end{thm}

Although the problem for semisimple groups has not yet been settled in complete generality, a fundamental theorem of the Fields Medallist G.\,A.\,Margulis settled most cases. In particular:

\begin{thm}[Margulis Superrigidity Theorem] \label{MargSuperThm}
   If $n \ge 3$, then all lattices in $\SL(n,\real)$ are superrigid.
\end{thm}

\begin{rem} \ 
\begin{enumerate}
\item The assumption that $n \ge 3$ is necessary: \emph{no} lattice in $\SL(2,\real)$ is superrigid. For example, if we let $H$ be any finite-index subgroup of $\SL(2,\integer)$, then $H$ is a lattice in $\SL(2,\real)$. However, it is possible to choose $H$ to be a free group, in which case $H$ has countless homomorphisms into $\SL(2,\real)$. Some of these homomorphisms have kernels that are infinite, but the kernel of any nontrivial homomorphism defined on $\SL(2,\real)$ must be finite.
\item Margulis proved superrigidity of lattices not only in $\SL(n,\real)$, but also in any simple Lie group~$G$ satisfying the technical condition that $\Rrank G \ge 2$.
\item The astute reader may have noticed that the modifier ``strictly'' is not being applied to ``superrigid'' in this section (c.f.\ Remark~\ref{NotStrictRem}). Although they are always superrigid, some lattices in $\SL(n,\real)$ are not strictly rigid.
\end{enumerate}
\end{rem}

Superrigidity implies that there is a very close connection between $H$ and~$G$. In fact, the connection is so close that it provides quite precise information on how to obtain~$H$ from~$G$.
Namely, superrigidity implies that letting $H$ be the integer points of~$G$ is often the only way to construct a lattice.

\begin{defn}
Suppose $H$ is a lattice in $G = \SL(n,\real)$. To avoid complications, let us assume $H$ is \emph{not} cocompact. We say $H$ is \emph{arithmetic} if there is an embedding of~$G$ in $\SL(d,\real)$, for some~$d$, such that $H$ is virtually equal to $G \cap \SL(d,\integer)$.
\end{defn}

\begin{thm}[Margulis Arithmeticity Theorem] \label{MargArithThm}
   If $n \ge 3$, then every lattice in $\SL(n,\real)$ is arithmetic.
\end{thm}

For convenience, we stated the arithmeticity theorem only for $\SL(n,\real)$, but it is valid for lattices in any simple Lie group~$G$ with $\Rrank G \ge 2$. It is a truly astonishing result.

\section{Why Superrigidity Implies Arithmeticity}

It is not at all obvious that superrigidity has anything to do with arithmeticity, so let us give some idea of how the connection arises. We warn the reader in advance that our motivation here is pedagogical rather than logical --- the main ideas in the proof of the Margulis
Arithmeticity Theorem \pref{MargArithThm} will be presented, but there will be no attempt to be rigorous.

We are given a lattice~$\Gamma$ in $G = \SL(n,\real)$, with $n \ge 3$, and we wish to show that $\Gamma$ is arithmetic.  Roughly speaking, we wish to show $\Gamma \subset \SL(n,\integer)$.

Here is a loose description of the 4 steps of the proof:
	\begin{enumerate}
	\item The Margulis Superrigidity Theorem \pref{MargSuperThm} implies that every matrix entry of every element of~$\Gamma$ is an algebraic number.
	\item Algebraic considerations allow us to assume that these algebraic numbers are rational; that is, $\Gamma \subset \SL(n,\rational)$.
	\item For every prime~$p$, a ``$p$-adic'' version of the Margulis Superrigidity Theorem provides a natural number~$N_p$, such that no element of~$\Gamma$ has a matrix entry whose denominator is divisible by~$p^{N_p}$. 
	\item This implies that some finite-index subgroup~$\Gamma'$ of~$\Gamma$ is contained in $\SL(n,\integer)$.
	\end{enumerate}

\setcounter{step}{0}

\begin{step} \label{ArithThmPf-algic}
 Every matrix entry of every element of\/~$\Gamma$ is an
algebraic number.
 \end{step}
 Suppose some $\gamma_{i,j}$ is transcendental.
 Then, for any transcendental number~$\alpha$, there is a
field automorphism~$\phi$ of~$\complex$ with
$\phi(\gamma_{i,j}) = \alpha$. Applying~$\phi$ to all the
entries of a matrix induces an automorphism~$\widetilde\phi$
of $\SL(n,\complex)$. Let
	$$ \text{$\varphi$ be the restriction of~$\widetilde\phi$ to~$\Gamma$,} $$
so $\varphi$ is a homomorphism from~$\Gamma$ to $\SL(n,\complex)$.
The Margulis Superrigidity Theorem \pref{MargSuperThm} implies there is a
continuous homomorphism $\hat\varphi \colon G \to
\SL(n,\complex)$, such that $\hat\varphi = \varphi$ on
a finite-index subgroup of~$\Gamma$.
(For simplicity, we have ignored the distinction between ``superrigid''
and ``strictly superrigid.'') Ignoring a finite group,
let us assume $\hat\varphi = \varphi$ on
all of~$\Gamma$.

Since there are uncountably many transcendental
numbers~$\alpha$, there are uncountably many different
choices of~$\phi$, so there must be uncountably many
different $n$-dimensional representations~$\hat\varphi$
of~$G$. However, it is well known from the the theory of
``roots and weights'' that $G$ (or any
connected, simple Lie group) has
only finitely many non-isomorphic representations
of any given dimension, so this is a contradiction.

{\smaller[2]
\begin{techrem}
Actually, this is not quite a
contradiction, because it is possible that two different
choices of~$\varphi$ yield the same representation of~$\Gamma$,
up to isomorphism; that is, after a change of basis. The
trace of a matrix is independent of the basis, so the
preceding argument really shows that the trace
of~$\varphi(\gamma)$ must be algebraic, for every
$\gamma \in \Gamma$. Then one can use some algebraic methods
to construct some other matrix representation~$\varphi'$
of~$\Gamma$, such that the matrix entries of~$\varphi'(\gamma)$
are algebraic, for every $\gamma \in \Gamma$.
\end{techrem}
}

\begin{step}
 We have $\Gamma \subset \SL(n,\rational)$.
 \end{step}
 Let $F$ be the subfield of~$\complex$ generated by the 
matrix entries of the elements of~$\Gamma$, so $\Gamma 
\subset \SL(n,F)$. From Step~\ref{ArithThmPf-algic}, we know
that this is an algebraic extension of~$\rational$.
Furthermore, because it is known that $\Gamma$ has a finite 
generating set, we see that this field extension is finitely
generated. Thus, $F$ is finite-degree field extension
of~$\rational$ (in other words, $F$ is an ``algebraic number
field''). This means that $F$ is almost the same
as~$\rational$, so it is only a slight exaggeration to say
that we have proved $\Gamma  \subset \SL(n,\rational)$.

Indeed, there is an algebraic technique, called ``Restriction of Scalars''
that provides a way to change $F$
into~$\rational$: there is a representation $\rho \colon G
\to \SL(\ell,\complex)$, such that $\rho \bigl( G \cap
\SL(n,F) \bigr) \subset \SL(\ell,\rational)$. Thus, after
changing to this new representation of~$G$, we have the
desired conclusion (without any exaggeration).

\begin{step} \label{MargArithPf-BddPowerP}
For every prime~$p$, there is a natural number~$N_p$, such 
that no element of~$\Gamma$ has a matrix entry whose 
denominator is divisible by~$p^{N_p}$. 
\end{step}
The fields $\real$ and~$\complex$ are complete (that is,
every Cauchy sequence converges), and they obviously 
contain~$\rational$. For any prime~$p$, the $p$-adic 
numbers~$\rational_p$ are another field that has these
same properties.

The Margulis Superrigidity Theorem~\pref{MargSuperThm}
deals with homomorphisms into $\SL(d,\F)$, where $\F = \real$,
but Margulis also proved a version of the theorem
that applies when $\F$ is a $p$-adic field. 
Now $G$ is connected,
but $p$-adic fields are totally disconnected, so every continuous 
homomorphism from~$G$ to $\SL(n,\rational_p)$ is trivial.
Thus, superrigidity tells us that $\varphi$ is trivial, up to a 
bounded error (c.f.\ Remark~\ref{NotStrictRem}). In other words, the closure of 
$\varphi(\Gamma)$ is compact in $\SL(n,\rational_p)$.

This conclusion can be rephrased in more elementary terms,
without any mention of the field $\rational_p$ of $p$-adic
numbers. Namely, it says that there is a bound on the 
highest power of~$p$ that divides any matrix entry of
any element of~$\Gamma$. This is what we wanted.

\begin{step}
Some finite-index subgroup~$\Gamma'$ of~$\Gamma$ is 
contained in $\SL(n,\integer)$.
\end{step}
 Let $D \subset \natural$ be the set consisting of the
denominators of the matrix entries of the elements of
$\varphi(\Gamma)$. 

We claim there exists $N \in \natural$, such that every
element of~$D$ is less than~$N$.
Since $\Gamma$ is known to be finitely generated,
some finite set of primes $\{p_1,\ldots,p_r\}$ contains all
the prime factors of every element of~$D$. (If $p$~is in the
denominator of some matrix entry of $\gamma_1
\gamma_2$, then it must
appear in a denominator somewhere in either $\gamma_1$
or~$\gamma_2$.) Thus, every element of~$D$ is of the
form $p_1^{m_1} \cdots p_r^{m_r}$, for some $m_1,\ldots,m_r
\in \natural$. From Step~\ref{MargArithPf-BddPowerP}, 
we know $m_i < N_{p_i}$,
for every~$i$. Thus, every element of~$D$ is less than 
$p_1^{N_{p_1}} \cdots p_r^{N_{p_r}}$. This establishes
the claim.

 From the preceding paragraph, we see that 
 $\Gamma \subset \frac{1}{N!} \Mat_{n
\times n}(\integer)$. Note that if $N = 1$, then $\Gamma
\subset \SL(n,\integer)$. In general, $N$ is a finite 
distance from~$1$, so it should not be hard to believe
(and it can indeed be shown) that some 
finite-index subgroup of~$\Gamma$ must be contained
in $\SL(n,\integer)$.
\qed

\renewcommand{\refname}{Further Reading}

\end{document}

%% file: bigset.tex
\newcommand\bigset[2]{\left\{\, #1 
 \mathrel{\left| \vphantom {\left\{ #1 \mid #2 \right\} }
 \right.} #2 \,\right\} }